\documentclass[12pt,a4paper]{amsart}
\usepackage[top=35mm, bottom=30mm, left=30mm, right=30mm]{geometry}
\usepackage{mathptmx}
\usepackage{mathrsfs}
\usepackage{verbatim}
\usepackage{url}
\usepackage[all]{xy}
\usepackage{xcolor}
\usepackage[colorlinks=true,citecolor=blue]{hyperref}
\usepackage{amsmath,amsthm}
\usepackage{amssymb}
\usepackage{enumerate}
\usepackage{graphicx}
\usepackage{tikz}
\usetikzlibrary{arrows}

\newtheorem{thm}{Theorem}[section]
\newtheorem{cor}[thm]{Corollary}
\newtheorem{lem}[thm]{Lemma}
\newtheorem{prop}[thm]{Proposition}

\theoremstyle{definition}
\newtheorem{defin}[thm]{Definition}

\numberwithin{equation}{section}

\begin{document}

\title[An alternative for minimal group actions]{An alternative for minimal group actions on totally regular curves}

\author[E.~Shi]{Enhui Shi}
\address[E. Shi]{Soochow University, Suzhou, Jiangsu 215006, China}
\email{ehshi@suda.edu.cn}

\author[H.~Xu]{Hui Xu}
\address[H. Xu]{CAS Wu Wen-Tsun Key Laboratory of Mathematics, University of Science and
Technology of China, Hefei, Anhui 230026, China}

\email{huixu2734@ustc.edu.cn}

\author[X.~Ye]{Xiangdong Ye}
\address[X. Ye]
{CAS Wu Wen-Tsun Key Laboratory of Mathematics, University of Science and
Technology of China, Hefei, Anhui 230026, China}
\email{yexd@ustc.edu.cn}

\begin{abstract}
Let $G$ be a countable group and $X$ be a totally regular curve. Suppose that $\phi:G\rightarrow {\rm Homeo}(X)$ is a minimal action. Then we show an alternative: either the action is topologically conjugate to isometries on the circle $\mathbb S^1$ (this implies that $\phi(G)$ contains an abelian subgroup of index  at most 2), or  has a quasi-Schottky subgroup (this implies that $G$ contains the free nonabelian group $\mathbb Z*\mathbb Z$). In order to prove the alternative,   we get a new characterization of totally regular curves by means of the notion of measure; and prove an escaping lemma holding for any minimal group action on infinite compact metric spaces, which improves a trick in
  Margulis' proof of the alternative in the case that $X=\mathbb S^1$.
\end{abstract}

\date{\today}

\setcounter{page}{1}

\maketitle

\section{Introduction}

The notion of amenable group was first introduced by von Neumann, which forbids the existence of a paradoxical decomposition of a group. On
the contrary, the free nonabelian group $\mathbb Z*\mathbb Z$ admits such a decomposition. This is the core of Banach-Tarski's decomposition
of the sphere. So, the following question asked by von Neumann is very natural:  whether every nonamenable group contains $\mathbb Z*\mathbb Z$ (see \cite{Day}).
This question was answered  positively by J. Tits for linear groups:

\begin{thm}\cite[Theorem 1]{Tits}\label{Tits alternative}
A finitely generated linear group either contains a free nonabelian subgroup or has a solvable subgroup of finite index.
\end{thm}

In general, von Neumann's question has a negative answer; one may consult \cite{Lodha, Monod, Ol} for many counterexamples.
Now, Theorem \ref{Tits alternative} is known as the Tits alternative. An interesting question is which group has the Tits alternative.
Many important groups coming from geometry and topology were shown to have this property (see e.g. \cite[p. 545]{Dru}). Nevertheless,
the exact analogue of Tits alternative does not hold even for subgroups of $C^\infty$ diffeomorphism  group of $\mathbb S^1$ (see \cite{Ghys1}).
As a replacement of the Tits alternative, G. Margulis proved the following theorem which answered positively a conjecture proposed by Ghys.
One may see \cite{Navas, Ghys2} for a different proof of this theorem by Ghys.

\begin{thm}\cite[Theorem 3]{Margulis}\label{Margulis alternative}
Let a group $G$ act by homeomorphisms on $\mathbb S^1$. Then either there is a $G$-invariant probability measure on $\mathbb S^1$, or $G$ contains a free nonabelian subgroup.
\end{thm}

Recently, some people are interested in studying the alternative phenomena for group actions on curves (continua of one dimension). For example, it is implied
by several authors' work that every subgroup of a dendrite homeomorphism group either has a finite orbit or contains a free nonabelian group (see \cite{Duch, Malyutin, Glasner}). One may refer to \cite{AN, MN, Shi, SY1, SY2} for some related investigations in this direction.

The purpose of the paper is to establish an alternative for group actions on a class of curves which contains all dendrites and all graphs. Explicitly, we obtain the following
theorem.

\begin{thm}\label{main theo}
Let $G$ be a countable group and $X$ be a totally regular curve. Suppose that $\phi:G\rightarrow {\rm Homeo}(X)$ is a minimal action. Then either the action is topologically conjugate to isometries on the circle $\mathbb S^1$ (this implies that $\phi(G)$ contains an abelian subgroup of index  at most 2), or  has a quasi-Schottky subgroup (this implies that $G$ contains the free nonabelian group $\mathbb Z*\mathbb Z$).
\end{thm}

Here we remark that the minimality condition in the theorem is not very strict, as there are many natural examples of minimal group actions on curves coming from
geometry (see e.g. \cite{Bowditch, Minsky}).  Margulis also established an alternative for minimal group actions on the circle \cite[Theorem 2]{Margulis}. The proof of Theorem \ref{main theo} follows the same line as in \cite{Margulis}; however, since the topology of the concerned curves are more
complicated than that of the circle, we have to develop some topological and dynamical ideas to overcome the difficulties encountered.

During the process of the proof, we get a new characterization of totally regular curves by means of the notion of measure, which may have its own interest in continuum theory.

\begin{thm}\label{sub theorem}
A continuum $X$ is totally regular if and only if there is an atomless probability Borel measure $\mu$ on $X$ such that
for every subcontinuum sequence $(K_i)_{i=1}^\infty$ satisfying $\mu(K_i)\rightarrow 0$, we always have ${\rm {diam}}(K_i)\rightarrow 0$.
\end{thm}

We also get an escaping lemma holding for any minimal group action on infinite compact metric space. This lemma is an extension of a
trick used by Margulis in \cite{Margulis}, and we avoid using the Neumann's theorem in group theory as he did.

\begin{lem}[Escaping lemma]\label{escaping lemma}
Let $X$ be an infinite compact metric space and let a countable group $G$ act on $X$ minimally. Then for any countable subset
$C\subset X$ and any finite subset $F\subset X$, there always exits a sequence $(g_n)$ in $G$ and a finite set $K$ in $X$ such that
$g_nF\rightarrow K$ and $K\cap C=\emptyset$.
\end{lem}

\medskip
{\bf Note.} We always assume that the group $G$ appeared in this article is a countable discrete topological group; in particular, it is secondly countable and locally compact.

\medskip
The paper is organized as follows. In Section 2, we introduce some basic notions and facts which will be used in this article. Then we  show the existence of contractible neighborhoods for minimal and sensitive actions on regular curves with a point of finite order in Section 3, and we use size function to construct a measure on continua in Section 4. In Section 5, we give a new characterization of totally regular curves via measures. We prove the escaping lemma for minimal actions in Section 6. Finally, we conclude the alternative for minimal actions on totally regular curves in Section 7.

\section{Preliminaries}

In this section, we will introduce some notions, notations, and facts in continuum theory, measure theory, and the theory of dynamical system, which will be used in the sequel.

\subsection{Characterizations of totally regular continua and the circle} By a {\it continuum}, we mean a connected compact metric space. A continuum $X$ is said to be \textit{nondegenerate} if it is not a single point. We say that $X$ is a {\it Peano continuum} if it is locally connected.
If a continuum $X$ does not contain uncountably many mutually disjoint nondegenerate subcontinua, then $X$ is called \textit{Suslinian}.
A continuum $X$ is said to be
\begin{enumerate}
\item {\it regular} if every point $x\in X$ has an open neighborhood basis $\mathcal N_x$ each member of which has finite boundary;

\item \textit{totally regular} if for any countable subset $F$ of a continuum $X$, there is a basis $\mathcal{B}$ of open sets for $X$ such that for each $B\in\mathcal{B}$, $F\cap \partial_X(B)=\emptyset$ and the boundary $\partial_X(B)$ of $B$ is finite.
\end{enumerate}
It is known that graphs and dendrites are totally regular; Suslinian continua are of one dimension; and regular
curves are Peano continua of dimension $1$. 

There have been many equivalent characterizations of totally regular continua (see \cite[Theorem 7.5]{NTT}).
Now, we recall a characterization given by S. Eilenberg and O. Harrold. Let $(X,d)$ be a metric space. For any $\varepsilon>0$, let $$L_{\varepsilon}^1(X,d)=\inf\sum_{n=1}^\infty\text{diam}(X_i),$$
where the infimum is taken over all decompositions $X=X_1\cup X_2\cup X_3\cup\cdots$ of $X$ such that $\text{diam}(X_i)<\varepsilon$ for each $i=1,2,\cdots$. Then $L_{\varepsilon}^1(X,d)$ is non-decreasing with respect to $\varepsilon$. Let
$$ L^1(X,d)=\lim_{\varepsilon\rightarrow 0^+}L_{\varepsilon}^1(X,d).$$
 If $L^1(X,d)<\infty$, then we say that $(X,d)$ has {\it finite linear measure.}

\begin{thm}[\cite{Harrold}]\label{finite meas}
A continuum $X$ is totally regular if and only if it has finite linear measure.
\end{thm}

Let $X$ be a topological space and $x\in X$. Let $\beta$ be a cardinal number. We say that $x$ is of {\it order $\leq \beta$}, written ${\rm ord}(x)\leq\beta$,
provided that $x$ has an open neighborhood basis $\mathcal N$ such that the cardinality of the boundary of each $U\in \mathcal N$ is less than or equal to $\beta$;
if ${\rm ord}(x)\leq\beta$ but ${\rm ord}(x)\nleq\alpha$ for any $\alpha<\beta$, then we say that $x$ is of {\it order $\beta$};
if $\beta<\aleph_0$, we say that $x$ is of {\it finite order}. The following characterization of the simple closed curve is due to  W. Ayres.

\begin{thm}[\cite{Ayres}, Corollary 3]\label{same order}
The simple closed curve is the only continuum all of whose points are of the same finite order.
\end{thm}

A point $p$ in a continuum $X$ is called a \textit{local separating point} provided that there is some neighborhood $U$ of $p$ such that the connected component $C$ of $U$ containing $p$ is separated by $p$ in $U$. Using the concept of local separating, Whyburn gave the following equivalent characterization of totally regular curve.

\begin{thm}[\cite{Whyburn}, Theorem 4]\label{loc sep}
A continuum $X$ is totally regular if and only if every nondegenerate subcontinuum of $X$ contains uncountable local separating points.
\end{thm}

Together with the following Whyburn's theorem, we know that every totally regular continuum has a point of order $2$.

\begin{thm}[\cite{Whyburn}, Theorem 1]\label{finite order pt}
A continuum with uncountable local separating points has a point of order $2$.
\end{thm}

\subsection{Hyperspace and size functions}\label{hyperspace and size function}

Let $(X, d)$ be a compact metric space and set
\begin{eqnarray*}
&&2^{X}=\{A: A \text{ is a nonempty closed subset of } X\},\\
&&C(X)=\{A\in 2^{X}: A \text{ is connected}\}.
\end{eqnarray*}
Then $2^{X}$ is a compact metric space endowed with the Hausdorff metric $d_H$ and $C(X)$ is closed in $2^{X}$ (\cite[Theorems 4.13, 4.17]{Nad1}).
We call each of $(2^{X}, d_H)$ and $(C(X), d_H)$ the {\it hyperspace} of $X$.

\begin{defin}\label{size fun}\cite[4.33]{Nad1}
A continuous function $\tau: 2^{X}\rightarrow \mathbb{R}$ is said to be a \textit{size function} (or \textit{Whitney map}) on $2^{X}$ if
\begin{enumerate}
  \item if $A,B\in 2^{X}$ and $A\subsetneq B$, then $\tau(A)<\tau(B)$;
  \item $\tau(\{x\})=0$, for any $x\in X$.
\end{enumerate}
\end{defin}

\begin{lem}\label{whitney map}\cite[4.33]{Nad1}
For a compact metric space $(X, d)$,
\begin{enumerate}
  \item there exists a size function on $2^X$;
  \item for any size function $\tau$ on $2^X$, $A_n\rightarrow A$ if and only if for any $\varepsilon>0,$ there is $N>0$, such that for each $n\geq N$,
      \begin{displaymath}
      A_n\subseteq N(A,\varepsilon) ~~\text{and  }~ |\tau(A_n)-\tau(A)|<\varepsilon.
      \end{displaymath}
\end{enumerate}
\end{lem}

{Note that $N(A,\varepsilon)=\{x\in X: d(x, A)<\varepsilon\}$. Here we recall an explicit construction of size functions in \cite[4.33]{Nad1}, which will be used in the sequel. Choose a countable dense subset $D=\{x_1,x_2,x_3,\cdots\}$ of $X$. For each $i$, 
define $f_i: X\rightarrow [0,1]$ by
\[ f_i(x)=\frac{1}{1+d(x_i,x)}.\]
For any subset $A$ of $X$, set $\tau_i(A)=\text{diam}f_i(A)$. Then the function
\[ \tau(A)=\sum_{i=1}^{\infty} \frac{1}{2^{i}} \tau_i(A)\]
is a size function on $2^{X}$.

\begin{lem}\label{size function}
Let $\tau$ be the size function defined above. Then, for any $C,C_1,C_2,\cdots\in C(X)$ with $C\subseteq \bigcup_{k=1}^{\infty} C_k$, we have
\begin{displaymath} \tau(C)\leq \sum_{k=1}^\infty\tau(C_{k}).\end{displaymath}
\end{lem}
\begin{proof}
For each $i$, $f_i$ is continuous on $X$. So, $f_i(C), f_i(C_1),f_i(C_2),\cdots$ are connected in $[0,1]$. Let $\lambda$ be the Lebesgue measure on $[0,1]$. Thus
\begin{eqnarray*}
\tau_i(C)&=&\text{diam}(f_i(C))=\lambda(f_i(C))\\
&\leq& \sum_{k=1}^\infty\lambda(f_i(C_k))= \sum_{k=1}^\infty\text{diam}(f_i(C_k))\\
&=& \sum_{k=1}^\infty\tau_i(C_k).
\end{eqnarray*}
Hence $\tau(C)\leq \sum_{k=1}^\infty\tau(C_{k})$.
\end{proof}

\subsection{Metric outer measures}
For a nonempty set $\Omega$ let $\mathcal{P}(\Omega)$ denote the power set of $\Omega$, i.e., $\mathcal{P}(\Omega)=\{A: A\subseteq \Omega\}$.
Recall that a function $\mu$ defined on $\mathcal{P}(\Omega)$ is called an \textit{outer-measure} on $\Omega$ if it satisfies:
\begin{enumerate}
  \item $0\leq \mu(E)\leq +\infty$, for each subset $E$ of $\Omega$;
  \item $\mu(\emptyset)=0$;
  \item if $E_1\subseteq E_2$, then $\mu(E_1)\leq \mu(E_2)$;
  \item if $\{E_i\}$ is any sequence of subsets of $\Omega$, then
  \begin{displaymath} \mu\left(\bigcup_{i=1}^\infty E_i\right)\leq \sum_{i=1}^{\infty} \mu(E_i).\end{displaymath}
\end{enumerate}
Let $\mu$ be an outer-measure on $\Omega$. A subset $E$ of $\Omega$ is said to be \textit{$\mu$-measurable} if for all subsets $A,B$ with $A\subseteq E$ and $B\subseteq \Omega\setminus E$, we have
\begin{displaymath} \mu(A\cup B)=\mu(A)+\mu(B).\end{displaymath}

\begin{thm}\cite[Theorem 3]{Rogers}
If $\mu$ is an outer-measure on $\Omega$, then the system $\mathcal{M}$ of $\mu$-measurable sets is a $\sigma$-algebra and the restriction of $\mu$ to $\mathcal{M}$ is a measure on $\mathcal{M}$.
\end{thm}

A function $\tau$ defined on a class $\mathcal{C}$ of subsets of $\Omega$ will be called a\textit{ pre-measure} if
\begin{enumerate}
  \item $\emptyset\in\mathcal{C}$ and $\tau(\emptyset)=0$;
  \item $0\leq \tau(C)\leq +\infty$  for all $C$ in $\mathcal{C}$.
\end{enumerate}


\begin{thm}\cite[Theorems 15,16,19]{Rogers}\label{metric outermeas}
If $\tau$ is a pre-measure defined on a class $\mathcal{C}$ of subsets in a metric space $(X,d)$, then the set function
\begin{equation}
\mu(E)=\sup_{\delta>0}\mu_{\delta}(E)
\end{equation}
is an outer-measure on $X$, where
\begin{displaymath}
\mu_{\delta}(E)=\inf\left\{\sum_{i=1}^{\infty}\tau(C_i):~C_i\in\mathcal{C}, \rm{diam}(C_i)\leq \delta, E\subseteq\bigcup_{i=1}^\infty C_i\right\}.
\end{displaymath}
(We let $\mu_{\delta}(E)=+\infty$ if the infimum is taken over the empty set.) Moreover, all Borel sets are $\mu$-measurable.
\end{thm}

\subsection{Equicontinuity and sensitivity in minimal systems}

Let $G$ be a countable group and $X$ be a compact metric space with metric $d$.
 A {\it continuous action} of $G$ on $X$, written $G\curvearrowright X$, means
a group homomorphism $\phi:G\rightarrow {\rm{Homeo}}(X),$ where ${\rm{Homeo}}(X)$ is the homeomorphism group of $X$.
For brevity, we usually use $gx$ instead of $\phi(g)(x)$ for $g\in G$ and $x\in X$. For $A\subset X$ and $g\in G$, denote by $gA$ the set $\{gx: x\in A\}$.
We use $GA$ to denote the set $\cup_{g\in G}=gA$; if $GA=A$, then $A$ is called {\it $G$-invariant}.
For $x\in X$, the {\it orbit} of $x$ under the action $G\curvearrowright X$ is the set $\{gx: g\in G\}$, which is denoted by $O(x, G)$.
If for every $x\in X$, $O(x, G)$ is dense in $X$, then the action $G\curvearrowright X$ is said to be {\it minimal}. If $A\subset X$ is closed and $G$-invariant, and
the restriction action $G\curvearrowright A$ is minimal, then we call $A$ a {\it minimal set}. If $\phi(G)$
is an equicontinuous family in ${\rm{Homeo}}(X)$ with respect to the uniform convergence topology, then $G\curvearrowright X$ is
said to be {\it equicontinuous}; that is, for every $\epsilon>0$ there is a $\delta>0$ such that $d(gx, gy)<\epsilon$ for all $g\in G$, whenever
$d(x, y)<\delta$. The action $G\curvearrowright X$ is said to be {\it sensitive} if there is $c>0$ such that for every nonempty open set
$U$ in $X$, there is a $g\in G$ with ${\rm diam}(gU)>c$; $c$ is sad to be a {\it sensitivity constant}.

The following dichotomy is well known in the theory of dynamical system and easy to prove.

\begin{thm}\label{dichotomy}
Let $G$ be a group and $X$ be a compact metric space. Suppose the action $G\curvearrowright X$ is minimal. Then
$G\curvearrowright X$ is either equicontinuous or sensitive.
\end{thm}

Let $H$ be a compact metric topological group and $K$ be a closed subgroup of $H$. Then $\psi:H\rightarrow {\rm{Homeo}}(H/K)$ define by
$\psi(h)(gK)=hgK$ is a continuous action of $H$ on $H/K$, which is called the {\it left translation action} of $H$ on $H/K$. We use
$L(H/K)$ to denote the subgroup $\psi(H)$ of ${{\rm Homeo}} (H/K)$.

The following theorem is classical and can be seen in \cite{Auslander}.

\begin{thm}\label{euqicontinuous}
Let $G$ be a group and $X$ be a compact metric space. Suppose the action $\phi: G\rightarrow {\rm Homeo}(X)$ is minimal and equicontinuous. Then
there is a compact metric topological group $H$ and a closed subgroup $K$ of $H$ such that $\phi$ is topologically conjugate to left translations
on $H/K$; that is, there is a homeomorphism $h:X\rightarrow H/K$ and a group homomorphism $\gamma: G\rightarrow L(H/K)$ such that
$h(\phi(g)(x))=\gamma(g)(h(x))$ for all $g\in G$ and $x\in X$; in particular, $X$ is topologically homogenous.
\end{thm}

\subsection{Contractible neighborhoods and strong $\epsilon$-proximality} We will recall some notions used by Margulis in \cite{Margulis}, some ideas
of which are due to Furstenberg. Let a group $G$ act on a compact metric space $(X, d)$ and let ${\mathcal M}(X)$ denote the space of all Borel
probability measures on $X$ with the standard weak* topology. A subset $K$ of $X$ is said to be {\it $G$-contractible} if there is a sequence
$(g_n)$ in $G$ such that ${\rm diam}(g_nK)\rightarrow 0$. We call a measure $\mu\in {\mathcal M}(X)$ {\it $G$-contractible} if there is a sequence
$(g_n)$ in $G$ and $x\in X$ such that $g_n\mu\rightarrow \delta_x$, where $\delta_x$ is the Dirac measure at $x$. We say that the $G$-action on $X$ is
{\it strongly $\epsilon$-proximal} if every measure $\mu\in {\mathcal M}(X)$ with ${\rm diam(supp}(\mu))<\epsilon$ is $G$-contractible.

The following lemma is implied by an argument of Lebesgue number.

\begin{lem}\label{contractible}
If every point $x\in X$ has a $G$-contractible neighborhood, then the $G$-action on $X$ is strongly $\epsilon$-proximal for some $\epsilon>0$.
\end{lem}

The following lemma is Proposition 1(ii) in \cite{Margulis}.

\begin{lem}\label{finite support}
If the $G$-action on $X$ is strongly  $\epsilon$-proximal for some $\epsilon>0$, then for any $\mu\in {\mathcal M}(X)$, there is
a $\nu\in {\mathcal M}(X)$ with finite support and a sequence $(g_n)$ in $G$ such that $g_n\mu\rightarrow \nu$.
\end{lem}

From Lemma \ref{contractible} and Lemma \ref{finite support}, we immediately have

\begin{prop}\label{con fin}
If every point $x\in X$ has a $G$-contractible neighborhood, then for any $\mu\in {\mathcal M}(X)$, there is
a $\nu\in {\mathcal M}(X)$ with finite support and a sequence $(g_n)$ in $G$ such that $g_n\mu\rightarrow \nu$.
\end{prop}

\subsection{ Quasi-Schottky groups} (see Fig.1.)  A group $H$ with two generators acting on a topological space $X$ is said to be
{\it quasi-Schottky} if there are generators $h_1, h_2$ of $H$ and disjoint nonempty open sets $U_1, U_2, V_1, V_2$ and
$W$ of $X$ such that
$$
h_1(U_1\cup U_2\cup V_2\cup W)\subset U_1, \ \ \ h_1^{-1}(U_2\cup V_1\cup V_2\cup W)\subset V_1,
$$
and
$$
h_2(U_1\cup U_2\cup V_1\cup W)\subset U_2, \ \ \ h_2^{-1}(U_1\cup V_1\cup V_2\cup W)\subset V_2.
$$
By a ``ping-pong" argument of Tits \cite{Tits},  we know that $H$ is a free nonabelian group. In addition, for any $h\not=e$, $h(W)\cap W=\emptyset$;
this implies that the $H$ action on the open set  $HW$ is discrete.
\begin{figure}
\centering
 \begin{tikzpicture}
\tikzset{
    vertex/.style={circle,draw,minimum size=2.5em},
    edge/.style={->,> = latex'}
}

\node[vertex] (1) at (0,0) {$W$};
\node[vertex] (2) at (-3,0) {$U_1$};
\node[vertex] (3) at (3,0) {$V_1$};
\node[vertex] (4) at (0,2) {$V_2$};
\node[vertex] (5) at (0,-2) {$U_2$};

 \draw[edge] (1) -- (2) node[midway, above] {\footnotesize $h_1$};
 \draw[edge] (1) -- (3) node[midway, above] {\footnotesize$h_1^{-1}$};
 \draw[edge] (1) -- (4) node[midway, right] {\footnotesize$h_2^{-1}$};
 \draw[edge] (1) -- (5) node[midway, right] {\footnotesize$h_2$};
 \draw[edge] (2.415) -- (4.205)node[pos=.4, above, sloped, rotate=0] {\footnotesize$h_2^{-1}$};
 \draw[edge] (4.229) -- (2.388)node[pos=.4, below, sloped, rotate=0]{\footnotesize$h_1$}  ;
 \draw[edge] (2.655) -- (5.155)node[pos=.4, above=4, sloped, rotate=0]{\footnotesize$h_1$} ;
 \draw[edge] (5.125) -- (2.685)node[pos=.4, below=4, sloped, rotate=0]{\footnotesize$h_2$}  ;
 \draw[edge] (3.605) -- (5.395)node[pos=.4, above=4, sloped, rotate=0]{\footnotesize$h_1^{-1}$}  ;
 \draw[edge] (5.425) -- (3.575)node[pos=.4, below=4, sloped, rotate=0]{\footnotesize$h_2$} ;
\draw[edge] (3.145) -- (4.300)node[pos=.4, above=4, sloped, rotate=0] {\footnotesize$h_1^{-1}$};
 \draw[edge] (4.330) -- (3.115)node[pos=.4, below=4, sloped, rotate=0]{\footnotesize$h_2^{-1}$}  ;
\path (2) edge [loop left] [below]node {\footnotesize$h_1$} (2)
(3) edge [loop right] node[below] {\footnotesize$h_1^{-1}$} (3)
(4) edge [loop above] node [right]{\footnotesize$h_2^{-1}$} (4)
(5) edge [loop below] node[right] {\footnotesize$h_2$} (5);

\end{tikzpicture}
\caption{Quasi-Schottky group}
\end{figure}
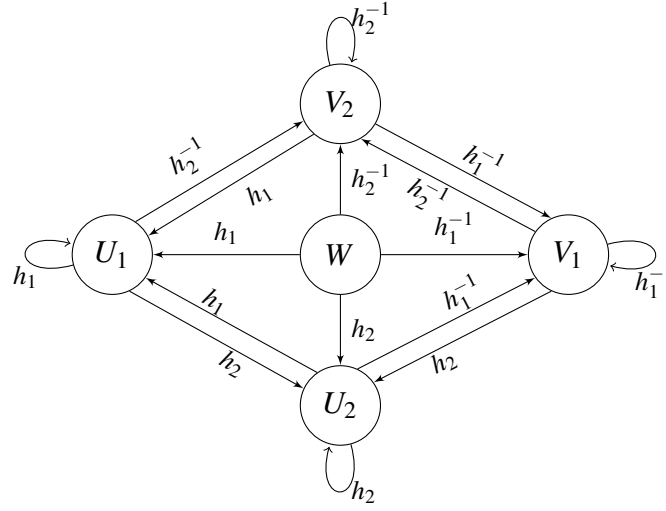


\section{Existence of contractible neighborhoods}

The following theorem is known as the Boundary Bumping Theorem (see e.g. \cite[p. 73]{Nad1}).

\begin{thm}\label{Boundary bumping}
Let $X$ be a continuum and let $U$ be a nonempty proper open subset of $X$. If $K$ is a component of $\overline U$, then $K\cap \partial_X(U)\not=\emptyset.$
\end{thm}

The following two lemmas are taken from  \cite{WSXX}. As the paper has not yet been officially published, for the convenience of the readers, we repeat the proof again here.

\begin{lem}\cite[Lemma 3.1]{WSXX}\label{large subcontinua}
Let $X$ be a regular curve and $U$ be a connected open subset of $X$ with $|\partial_X(U)|=n$ for some positive integer $n$.
If ${\rm diam}(U)>\epsilon$ for some $\epsilon>0$, then there is some connected open set $V\subset U$, such that $d(V, \partial_X(U))\geq \epsilon/4n$
and ${\rm diam}(V)\geq \epsilon/4n$.
\end{lem}

\begin{proof}
Let $\partial_X(U)=\{e_1, e_2, \cdots, e_n\}$. We claim that $U\setminus\bigcup_{i=1}^{n}B(e_i,\epsilon/2n)\neq\emptyset$. In fact, if $U\subset\bigcup_{i=1}^{n}B(e_i,\epsilon/2n)$, then for any two distinct points $a,b\in U$, by the connectivity of $U$ there are finite $B(e_{i_1},\epsilon/2n), \cdots, B(e_{i_m},\epsilon/2n)$ ($m\leq n$) such that $a\in B(e_{i_1},\epsilon/2n)$, $b\in B(e_{i_m},\epsilon/2n)$ and $B(e_{i_k},\epsilon/2n)\cap B(e_{i_{k+1}},\epsilon/2n)\neq\emptyset$ for $1\leq k\leq m-1$. Thus we have
\begin{align*}
d(a,b)\leq & d(a,e_{i_1})+d(e_{i_1},e_{i_2})+\cdots +d(e_{i_{m-1}},e_{i_m})+d(e_{i_m},b)\\
< & \frac{\epsilon}{2n}+\frac{\epsilon}{n}\cdot(m-1)+\frac{\epsilon}{2n}\\
\leq & \epsilon.
\end{align*}
It follows that ${\rm diam}(U)\leq\epsilon$, which is a contradiction. Hence $U\setminus\bigcup_{i=1}^{n}B(e_i,\epsilon/2n)\neq\emptyset$.

Take a point $x\in U\setminus\bigcup_{i=1}^{n}B(e_i,\epsilon/2n)$. Then $d(x,\partial_X(U))\geq \epsilon/2n$, and hence
$$d(B(x,\epsilon/4n), \partial_X(U))\geq \epsilon/4n.$$
 Let $V$ be the component of $B(x,\epsilon/4n)\cap U$ which contains $x$. Then $d(V,\partial_X(U))\geq\epsilon/4n$. Let $W=U\cap B(x,\epsilon/4n).$
Since $U$ is connected, $\emptyset\not= \partial_X(W)\subset\partial_X(B(x,\epsilon/4n))$.
This together with Theorem \ref{Boundary bumping} imlies $\emptyset\not= \partial_X(V)\subset \partial_X(B(x,\epsilon/4n))$.
 So, ${\rm diam}(V)={\rm diam}(\overline{V})\geq \epsilon/4n$.
\end{proof}

\begin{lem}\cite[Proposition 3.2]{WSXX}\label{convergence continua}
Let $X$ be a regular curve and let $(U_i)_{i=1}^\infty$ be a sequence of connected open subsets of $X$ with $|\partial_X(U_i)|=n$ for some positive integer $n$ and for each $i$. Suppose that there is some $\epsilon>0$ with ${\rm diam}(U_i)>\epsilon$ for each $i$.  Then there is a nonempty open subset $W$ of $X$ and infinitely many $i$'s such that $W$ is contained in $U_i$.
\end{lem}

\begin{proof}
For each $i$, it follows from Lemma~\ref{large subcontinua} that there is a connected open subset $W_i\subset U_i$ with $d(W_i, \partial_X(U_i))\geq \epsilon/4n$ and ${\rm diam}(W_i)\geq \epsilon/4n$. By the compactness of $2^X$ and $C(X)$, there are subsequences $(\overline{W_{i_k}})$ and $(\partial_X(U_{i_k}))$ such that $(\overline{W_{i_k}})$ converges to a subcontinuum $A$, and
$$d_H(\partial_X(U_{i_{k_1}}), \partial_X(U_{i_{k_2}}))<\epsilon/4n, \forall k_1\not=k_2.\eqno(1)$$
Take a point $z\in A$. Then there exists a connected open neighborhood $Q$ of $z$ such that $\partial_X(Q)$ is finite and ${\rm diam}(Q)<\epsilon/4n$. Since $(\overline{W_{i_k}})$ converges to $A$, there exists a positive integer $N$ such that $W_{i_k}\cap Q\neq\emptyset$ for each $k\geq N$. Noting that  ${\rm diam}(W_{i_k})\geq \epsilon/4n$ and ${\rm diam}(Q)<\epsilon/4n$, we have $W_{i_k}\nsubseteq Q$. Thus $W_{i_k}\cap \partial_X(Q)\neq\emptyset$ by the connectivity of $W_{i_k}$.
Hence, there exist a point $p\in \partial_X(Q)$ and infinitely many $k$'s such that $p\in W_{i_k}$. Passing to a subsequence if necessary, we may suppose that $p\in W_{i_k}$ for each $k\geq N$.

Let $W=W_{i_N}$. To complete the proof, we only need to show that $W\subset U_{i_k}$ for all $k\geq N$. Otherwise, there is some $k'\geq N$ with $W\nsubseteq U_{i_{k'}}$.  Since $W$ is connected and $p\in W\cap U_{i_{k'}}$, there is a point $e\in W\cap\partial_X(U_{i_{k'}})$.  By $(1)$,  there is $e'\in \partial_X(U_{i_N})$ such that $d(e,e')<\epsilon/4n$. Then we have $d(W, e')\leq d(e,e')<\epsilon/4n$. This contradicts the assumption that $d(W,\partial_X(U_{i_N}))\geq\epsilon/4n$ at the beginning.
\end{proof}

\begin{thm}\label{contract neighbor}
Let $X$ be a regular curve with a point of finite order and let a group $G$ act on $X$ minimally and sensitively. Then every point $x\in X$ has a contractible neighborhood.
\end{thm}

\begin{proof}
 Let $x$ be a point of finite order in $X$. Then there is a positive integer $n$ and a sequence of open neighborhoods $(U_i)_{i=1}^\infty$ of $x$
such that $|\partial_X(U_i)|=n$ and ${\rm diam}(U_i)\rightarrow 0$ as $i\rightarrow\infty$. Let $c>0$ be a sensitivity constant of the action.
Then for each $i$, there is $g_i\in G$ such that ${\rm diam}(g_i(U_i))>c$. By Lemma \ref{convergence continua}, there exists a nonempty open set
$W$ in $X$ such that $W$ is contained in infinitely many $U_i$'s. We may as well assume that $W\subset U_i$ for each $i$.
Then ${\rm diam}(g_i^{-1}W)\leq {\rm diam}(U_i)\rightarrow 0$. For any $y\in X$, by the minimality of the action, there is some $g\in G$ with
$gy\in W$. Take an open neighborhood $V$ of $y$ such that $gV\subset W$. Then $V$ is a contractible neighborhood of $y$ as
 ${\rm diam}(g_i^{-1}gV)\rightarrow 0 \ (i\rightarrow \infty)$.
\end{proof}

\section{Measures induced by size functions}

Let $X$ be a continuum. Let $\tau $ be the size function on $2^{X}$ as defined in Section \ref{hyperspace and size function}. Set $\mathcal{C}=C(X)\cup\{\emptyset\}$ and $\tau(\emptyset)=0$. Let $\mu$ be the outer-measure defined as in Theorem \ref{metric outermeas}  by
\begin{equation}\label{meas mu}
\mu(E)=\sup_{\delta>0}\mu_{\delta}(E),
\end{equation}
 where
\begin{displaymath}
\mu_{\delta}(E)=\inf\left\{\sum_{i=1}^{\infty}\tau(C_i):~C_i\in\mathcal{C}, \rm{diam}(C_i)\leq \delta, E\subseteq\bigcup_{i=1}^\infty C_i\right\}.
\end{displaymath}
Here, we should note that $\mu_\delta(E)=+\infty$ if the infimum is taken over an empty set. By Theorem \ref{metric outermeas}, the restriction of $\mu$ to the Borel $\sigma$-algebra is a measure.

\begin{lem}\label{atomless of mu}
$\mu$ is atomless.
\end{lem}
\begin{proof}
For any $x\in X$, it is clear $\tau(\{x\})=0$. By the definition of $\mu$, we have $\mu(\{x\})=0$.
\end{proof}

\begin{lem}\label{meas and size}
For any $C\in C(X)$, we have $\mu(C)\geq\tau (C)$.
\end{lem}
\begin{proof}
If for some $\delta>0$, there is no sequence $(C_i)$ in $\mathcal{C}$ with $\rm{diam}(C_i)\leq \delta$ and $C\subseteq\cup C_i$,
then
 $$\mu(C)\geq \mu_\delta(C)=+\infty\geq \tau(C).$$
Otherwise, for every $\delta>0$ and $\epsilon>0$, there is a sequence $(C_i)$ in $\mathcal{C}$ with $\rm{diam}(C_i)\leq \delta$ and $C\subseteq\cup C_i$ such that
\[ \tau(C)\leq \sum_{i=1}^{\infty} \tau (C_i)\leq \mu_\delta(C)+\epsilon,\]
by Lemma \ref{size function}. By the arbitrariness of $\epsilon$ and the definition of $\mu$, we still have $\mu(C)\geq\tau (C)$.
\end{proof}

\begin{lem}\label{general 0}
For any continuum $X$, there exists an atomless Borel measure $\mu$ on $X$ such that for any sequence $(Y_n)$ of subcontinua, if $\mu(Y_n)\rightarrow 0$, then ${\rm{diam}} (Y_n)\rightarrow 0$.
\end{lem}
\begin{proof}
Let $\mu$ be the measure defined as (\ref{meas mu}). If ${\rm{diam}} (Y_n)$ does not converge to $0$, then there exist a subsequence $(Y_{n_k})$ and $Y\in C(X)$
with ${\rm diam}(Y)>\delta$ for some $\delta>0$, such that $Y_{n_k}\rightarrow Y$. By Lemma \ref{meas and size},  $\mu(Y_{n_k})\geq \tau(Y_{n_k})\rightarrow 0$. Since $\tau$ is continuous on $2^{X}$, we have $\tau(Y)=0$. By the Definition \ref{size fun} of the size function, $Y$ must be degenerate. This is a contradiction.
\end{proof}

{\bf Note.} Here we should note that the measure $\mu$ in Lemma \ref{general 0} may take the infinite value at some Borel subsets.

\section{A characterization of totally regular curves}

Before the proof of the main result in this section, let us first recall and show some lemmas in continuum theory.

\begin{lem}\cite[7.21]{Nad1}\label{ULAC}
A Peano continuum $(X,d)$ is ULAC (uniformly  locally arcwise connected). That is for any $\varepsilon>0$ there exists $\delta>0$ such that if $d(x,y)<\delta$ and $x\neq y$, then there is an arc $A\subset X$ such that $A$ has end points $x$ and $y$ and $\text{diam}(A)<\varepsilon$.
\end{lem}

The following lemma is an easy conclusion of the regularity of Borel measures on metric spaces.

\begin{lem}\label{atomless}
Let $\mu$ be an atomless Borel probability measure on a continuum $(X,d)$. Then for any sequence of subcontinua $(K_n)$ of $X$ with $\lim_{n\rightarrow \infty}\text{diam}(K_n)=0$, we have $\lim_{n\rightarrow \infty}\mu(K_n)=0$.
\end{lem}

\begin{lem}\label{peano}
Let $X$ be a Peano continuum. Let $V$ be a connected open subset and $F$ be a closed subset of $X$. If $F\subset V$, then there is a subcontinuum $K$ of $X$ satisfying $F\subset K\subset V$.
\end{lem}
\begin{proof}
Since $X$ is locally connected, for every $x\in F$, there is a connected neighborhood $W_x$ of $x$ such that $\overline{W_x}\subset V$. By the compactness,  there exist $x_1,\cdots,x_n\in F$ such that $F\subset W_{x_1}\cup\cdots\cup W_{x_n}$. Noting that every connected open set of a Peano continuum is arcwise connected, for each $i\in\{1,\cdots,n-1\}$, there is an arc $A_i\subset V$ connecting $B_{x_i}$ and $B_{x_{i+1}}$. Let
\[ K= \overline{(\cup_{i=1}^n B_{x_i})\cup (\cup_{i=1}^{n-1} A_i)}.\]
Then $K$ is a subcontinuum satisfying $F\subset K\subset V$.
\end{proof}

The following definition describes a kind of $1$-dimensionality of a continuum from the viewpoint of measure.

\begin{defin}\label{one dim}
A continuum $(X,d)$ is said to be of \textit{$1$-dimension in the sense of measure} if there is an atomless Borel probability measure $\mu$
on $X$ such that for any sequence of subcontinua $(Y_n)$ of $X$ with $\mu(Y_n)\rightarrow 0$, we always have $\text{diam}(Y_n)\rightarrow 0$.
\end{defin}

It can be seen from the following lemma that a continuum of $1$-dimension in the sense of measure is always a Peano continuum.
\begin{lem}\label{1-dim}
If a continuum $X$ is of $1$-dimension in the sense of measure, then it is a Peano continuum.
\end{lem}
\begin{proof}
If $X$ is not a Peano continuum, then there exists a convergence continuum (\cite[5.12]{Nad1}), i.e., there is a nondegenerate subcontinuum $A$ and a sequence $(A_n)_{n=1}^{\infty}$ of mutually disjoint subcontinua $A_n$ such that
\[ A=\lim_{n\rightarrow\infty}A_n~~\qquad\text{  and   }~~~ \qquad A\cap A_n=\emptyset~~\text{ for each}~~n.\]
Let $\mu$ be a probability measure on $X$ such that for any sequence of subcontinua $(Y_n)$ of $X$ with $\mu(Y_n)\rightarrow 0$, we always have $\text{diam}(Y_n)\rightarrow 0$.
Since $A_n$'s are pairwise disjoint and $\mu(X)=1$, we have $\mu(A_n)\rightarrow 0$ as $n\rightarrow\infty$. Then ${\rm diam}(A_n)\rightarrow 0$ by the assumption on $\mu$,
which is a contradiction.
\end{proof}

It is well known that the topological dimension of Suslinian continua are one. However, we did not find a reference giving an explicit proof. For the convenience of the readers, we afford a sketched proof here.
\begin{lem}\label{dim of suslinian}
Every Suslinian continuum is of topological dimension $1$. 
\end{lem}
\begin{proof}
Let $X$ be a Suslinian continuum with metric $d$ and fix $x\in X$.  Set
\[ E=\{r\in (0,\text{diam}(X)]:~\partial B(x,r) \text{ has a non-degenerate component}\}.\]
Then $E$ is countable by the definition of Suslinian continuum. Thus, $\{B(x,r):r\in (0,\text{diam}(X)]\setminus E\}$ forms a open neighborhood of $x$ with $0$-dimensional boundary. Hence $\text{dim}(X)\leq 1$. Note that the dimension of every non-degenerate continuum is no less than $1$. So $\text{dim}(X)=1$.
\end{proof}

Now we start to state and prove the main theorem in this section which gives a new characterization of totally regular curves.

\begin{thm}\label{total dim}
A continuum $(X,d)$ is totally regular if and only if it is $1$-dimensional in the sense of measure.
\end{thm}

\begin{proof}($\Longleftarrow$)
Suppose that $(X,d)$ is $1$-dimensional in the sense of measure, then there is an atomless Borel probability measure $\mu$
on $X$ such that for any sequence of subcontinua $(Y_n)$ of $X$ with $\mu(Y_n)\rightarrow 0$, we have $\text{diam}(Y_n)\rightarrow 0$.
Define a function $\rho: X\times X\rightarrow \mathbb{R}$ by
\[\rho(x,y)=\inf\{\mu(K):~K ~~\text{is a subcontinuum of}~~X~~\text{containing }~~ x~~\text{and}~~ y\}.\]

\medskip
\noindent\textbf{Claim 1}. $\rho$ is a metric on $X$. \\

It is clear that we only need to show the triangle inequality and the requirement that $\rho(x,y)=0$ implies $x=y$.

Assume that $\rho(x,y)=0$. Then there is a sequence $(K_n)$ of subcontinua of $X$ containing $x$ and $y$ such that $\lim_{n\rightarrow\infty}\mu(K_n)=0$. By the assumption on $\mu$, we have $\lim_{n\rightarrow\infty}\text{diam}(K_n)=0$. Thus $d(x,y)=0$, and hence $x=y$.

For any $x,y,z\in X$ and any $\varepsilon>0$, there exist subcontinua $K_{x,y}$ and $K_{y,z}$ of $X$ containing $\{x,y\}$ and $\{y,z\}$ respectively such that
\[\mu(K_{x,y})<\rho(x,y)+\frac{\varepsilon}{2},~~\text{and}~~~\mu(K_{y,z})<\rho(y,z)+\frac{\varepsilon}{2}.\]
Then $K\equiv K_{x,y}\cup K_{y,z}$ is a subcontinuum of $X$ containing $x$ and $z$ such that
\[\rho(x,z)\leq\mu(K)\leq \mu(K_{x,y})+\mu(K_{y,z})\leq \rho(x,y)+\rho(y,z)+\varepsilon.\]
Since $\varepsilon $ is arbitrary, we have $\rho(x,z)\leq\rho(x,y)+\rho(y,z)$. Thus  $\rho$ satisfies the triangle inequality.\\

\noindent \textbf{Claim 2}. The metric $\rho$ is compatible with $d$. \\

We need to show that $\lim_{n\rightarrow\infty}d(x_n,x)=0$ is equivalent to $\lim_{n\rightarrow\infty}\rho(x_n,x)=0$ for any sequence $(x_n)$ of $X$.

On the one hand, suppose $\lim_{n\rightarrow\infty}\rho(x_n,x)=0$. Then for any $\varepsilon>0$, there exists a positive integer $N$ such that for any $n\geq N$, $\rho(x_n,x)<\frac{\varepsilon}{2}$. By the definition of $\rho$, for each $n\geq N$, there is a subcontinuum $K_n$ containing $x_n$ and $x$ such that $\text{diam}(K_n)\leq \rho(x_n,x)+\frac{\varepsilon}{2}<\varepsilon$. Thus
\[d(x_n,x)\leq \text{diam}(K_n)<\varepsilon.\]
 Hence $\lim_{n\rightarrow\infty}d(x_n,x)=0$.

On the other hand, suppose $\lim_{n\rightarrow\infty}d(x_n,x)=0$. Lemma \ref{1-dim} implies that $X$ is a Peano continuum. Fix $\varepsilon>0$. Then by Lemma \ref{ULAC}, there exists $\delta>0$ such that if $d(x,y)<\delta$ and $x\neq y$, then there is an arc $A\subset X$ such that $A$ has end points $x$ and $y$ and $\text{diam}(A)<\varepsilon$. Since $d(x_n,x)\rightarrow 0$, there exists $N$ such that for any $n\geq N$, $d(x_n,x)<\delta$. Therefore, for any $n\geq N$, there is an  arc $A_n$ connecting $x_n$ and $x$ and $\text{diam}(A_n)<\varepsilon$. Thus there exists a sequence $(K_n)$ of subcontinua of $X$ containing $x_n$ and $x$ satisfying $\lim_{n\rightarrow\infty}\text{diam}(K_n)=0$. By Lemma \ref{atomless},  $\lim_{n\rightarrow\infty}\mu(K_n)=0$. By the definition of $\rho$, we have $\rho(x_n,x)\leq \mu(K_n)$ for each $n$. Thus $\lim_{n\rightarrow\infty}\rho(x_n,x)=0$. \\

\noindent\textbf{Claim 3}. $X$ is of dimension $1$.\\

 By the definition of $\rho$, it is clear that $\text{diam}_{\rho}(K)\leq \mu(K)$, for any subcontinuum $K$ of $X$. Thus there are at most countable many pairwise disjoint nondegenerate subcontinua in $X$; otherwise $\mu(X)=\infty$, which is a contradiction. Then $X$ is Suslinian and hence has dimension $1$ by Lemma \ref{dim of suslinian}.\\

From Claim 3,  for any $\varepsilon>0$, there is a finite open cover $\mathcal{U}=\{U_1,\cdots,U_n\}$ with $\text{mesh}(\mathcal{U})<\varepsilon$ and $\text{ord}(\mathcal{U})\leq 1$ (Here $\text{ord}(\mathcal{U})=-1+\sup_{x\in X}|\{i\in\{1,\cdots,n\}: x\in U_i\}|$). For each $i\in\{1,\cdots,n\}$, let $\mathcal{U}_i$ be the set of connected components of $U_i$. By the local connectivity, the connected components of an open set are all open. Then $\bigcup_{i=1}^n \mathcal{U}_i$ is an open cover of $X$. By the compactness of $X$, there is a finite subcover $\mathcal{V}=\{V_1,V_2,\cdots, V_m\}$ of $X$. Moreover, $\text{mesh}(\mathcal{V})<\varepsilon$ and $V_i\cap V_j\cap V_k=\emptyset$ for any distinct $i,j,k\in\{1,\cdots,m\}$.

For each $i\in\{1,\cdots,m\}$, we can choose a closed subset $F_i$ contained in $V_i$ such that
$X=F_1\cup F_2\cup\cdots \cup F_m$ and $F_i\cap F_j\cap F_k=\emptyset$ for any distinct $i,j,k\in\{1,\cdots,m\}$. By Lemma \ref{peano}, there exists a subcontinuum $K_i$ satisfying $F_i\subset K_i\subset V_i$ for each $i$. It is clear that $\mathcal{K}\equiv\{K_i:i=1,\cdots,m\}$ also satisfy
\[X=K_1\cup K_2\cup\cdots \cup K_m,~~\text{mesh}(\mathcal{K})<\varepsilon,~~\text{and}~~K_i\cap K_j\cap K_k=\emptyset, \eqno{(\rm i)}\]
for any distinct $i,j,k\in\{1,\cdots,m\}$.
Set $\hat{K_i}=K_i\setminus(\cup_{j\neq i}K_j)$ for each $i$. Then, by ({\rm i}),
\begin{equation}\label{eq1}
\mu(K_i)=\mu(\hat{K_i})+\sum_{j\neq i}\mu(K_i\cap K_j).
\end{equation}
Note that $\text{diam}_{\rho}(K_i)\leq \mu(K_i)$ for each $i$. By (\ref{eq1}) and the definitions in Section 2.1,
\begin{eqnarray*}
L_{\varepsilon}^1(X,\rho)&\leq& \sum_{i=1}^m\text{diam}_{\rho}(K_i)\leq \sum_{i=1}^m\mu(K_i)\\
&=&\sum_{i=1}^m\left(\mu(\hat{K_i})+\sum_{j\neq i}\mu(K_i\cap K_j)\right)\\
&=& \sum_{i=1}^m\mu(\hat{K_i})+\sum_{i=1}^m\sum_{j\neq i}\mu(K_i\cap K_j)\\
&=& \mu\left(\cup_{i=1}^m \hat{K_i} \right) +\mu(\cup_{1\leq i<j\leq m}K_i\cap K_j )\\
&\leq&2.
\end{eqnarray*}
Since $\varepsilon$ is arbitrary, we have $L^1(X,\rho)\leq 2$. By Theorem \ref{finite meas}, $X$ is totally regular.

\medskip

$(\Longrightarrow)$ Now assume that $X$ is a totally regular curve. We may assume that the metric $d$ on $X$ is such that all open balls are connected (see \cite[8.50]{Nad1}). We claim  that the measure $\mu$ defined in (\ref{meas mu}) is finite.

Let $A$ be a subset of $X$. By the definitions in Section 2.2, for each positive integer $i$,
\begin{eqnarray*}
\tau_i(A)&=&\text{diam}(f_i(A))=\sup_{x,y\in A}\left|\frac{1}{1+d(x, x_i)}- \frac{1}{1+d(y, x_i)}\right|\\
&=& \sup_{x,y\in A}\left|\frac{d(x, x_i)-d(y, x_i)}{(1+d(x, x_i))(1+d(y, x_i))}\right|\\
&\leq &\sup_{x,y\in A}d(x, y)
\leq \text{diam}(A).
\end{eqnarray*}
Thus
\[\tau(A)=\sum_{i=1}^{\infty}\frac{1}{2^{i}}\tau_{i}(A)\leq \text{diam}(A).\]

For every $\delta>0$, let $X=X_1\cup X_2\cup X_3\cup\cdots$ be a decomposition  of $X$ with $\text{diam}(X_i)<\delta$ for each $i$. Note that the metric $d$ is such that all open balls are connected. Then each $X_i$ is contained in a connected closed ball $B_i$ with $\text{diam}(B_i)\leq 2\text{diam}(X_i)$. Thus
\begin{equation*}
\inf\left\{\sum_{i=1}^{\infty}\tau(C_i):~C_i\in\mathcal{C}, \rm{diam}(C_i)\leq \delta, E\subseteq\bigcup_{i=1}^\infty C_i\right\}\leq 2L^1_{\delta/2}(X,d).
\end{equation*}
Therefore
\begin{eqnarray*}
\mu(X)&=&\sup_{\delta>0}\mu_{\delta}(X)=\lim_{\delta\rightarrow 0^+}\mu_{\delta}(X)\\
&=&\lim_{\delta\rightarrow 0^+}\inf\left\{\sum_{i=1}^{\infty}\tau(C_i):~C_i\in\mathcal{C}, \rm{diam}(C_i)\leq \delta, E\subseteq\bigcup_{i=1}^\infty C_i\right\}\\
&\leq&\lim_{\delta\rightarrow 0^+}\inf\left\{\sum_{i=1}^{\infty}\text{diam}(C_i):~C_i\in\mathcal{C}, \rm{diam}(C_i)\leq \delta, E\subseteq\bigcup_{i=1}^\infty C_i\right\}\\
&\leq& \lim_{\delta\rightarrow 0^+}2L^1_{\delta/2}(X,d)\\
&=&2L^{1}(X,d)<\infty.
\end{eqnarray*}
This together with Lemma \ref{atomless of mu} and Lemma \ref{general 0} implies that $X$ is of $1$-dimension in the sense of measure.

\end{proof}

\section{An escaping lemma for minimal actions}

Let $G$ be a countable group. A subset $S$ of $G$ is said
to be {\it syndetic} if there is a finite set $F\subset G$ such that $G=FS$. Suppose $X$ is a compact metric space with metric $d$ and $G\curvearrowright X$ is a continuous action.
Then a point $x\in X$ is said to be {\it almost periodic} if for every neighborhood $U$ of $x$, the set $R(x, U)\equiv \{g\in G: gx\in U\}$ is syndetic.
The following theorem is well known (see e.g. \cite[Ch.1-Theorem 1]{Auslander}).

\begin{thm}\label{almost periodic}
Let $G\curvearrowright X$ be an action of a countable group $G$ on a compact metric space $X$. Then $x\in X$ is almost periodic if and only if
$\overline{Gx}$ is a minimal set.
\end{thm}

For $x\not= y\in X$, if there is a sequence $(g_n)$ in $G$ such that $d(g_nx, g_ny)\rightarrow 0$, then $x$ and $y$ are said to be {\it proximal}.
The following theorem can be seen in \cite[Ch.5-Theorem 3]{Auslander}.

\begin{thm}\label{prox almost}
Let $G\curvearrowright X$ be an action of a  group $G$ on a compact metric space $X$. Then for every $x\in X$, there is an almost periodic point
$x^*$ which is proximal to $x$.
\end{thm}

\begin{lem}\label{prox lemma}
Let $G$ be a group acting on a compact metric space $X$. Suppose $x_1, x_2,..., x_n$ are $n$ points in $X$. Then there are
$y_1, y_2,..., y_k$ in $X$ with $1\leq k\leq n$ and $(g_i)$ in $G$ such that $g_i\{x_1,x_2,...,x_n\}\rightarrow \{y_1,y_2,.., y_k\}$
and for any $p\not=q$, $y_p$ and $y_q$ are not proximal.
\end{lem}

\begin{proof}
Take $k$ to be the minimal positive integer such that there are $y_1, y_2,..., y_k$ and a sequence $(g_i)$ in $G$ with $g_i\{x_1,x_2,...,x_n\}\rightarrow \{y_1,y_2,.., y_k\}$.
Then for any $p\not=q$, $y_p$ and $y_q$ are not proximal.
\end{proof}

Now we state and prove the escaping lemma, which is key in constructing a ping-pong-game in next section.

\begin{lem}[Escaping lemma]\label{escaping lemma}
Let $X$ be an infinite compact metric space and let a countable group $G$ act on $X$ minimally. Then for any countable infinite set
$C\subset X$ and any finite set $F\subset X$, there always exit a sequence $(g_n)$ in $G$ and a finite set $K$ in $X$ such that
$g_nF\rightarrow K$ and $K\cap C=\emptyset$.
\end{lem}

\begin{proof} (see Fig.2.)
Write $C=\{c_i:i=1,2,3,...\}$ and $F=\{a_1, a_2,...,a_k\}$ for some positive integer $k$.  From Lemma \ref{prox lemma}, we may suppose that
$a_i$ and $a_j$ are not proximal for any $i\not=j$.
Consider the product action $G\curvearrowright X^k$ defined
by $$g(x_1, x_2,..., x_k)=(gx_1,gx_2,...,gx_k).$$
From Theorem \ref{prox almost}, there is an almost periodic point $(w_1, w_2,..., w_k)\in X^k$ which is proximal to
$(a_1, a_2,...,a_k)$. Then there is $(b_1, b_2,..., b_k)\in \overline{G(w_1,w_2,...,w_k)}\subset X^k$
and a sequence $(g_n)$ with $g_na_i\rightarrow b_i$ for each $1\leq i\leq k$, as $n\rightarrow\infty$.
By the non-proximality of each pair of $a_i$ and $a_j$, these $b_i$'s are pairwise distinct. Note that $(b_1, b_2,..., b_k)$ is
still almost periodic by Theorem \ref{almost periodic}. Set $B=\{b_1, b_2,..., b_k\}\subset X$; then
$g_nF\rightarrow B$. We can suppose that $c_1\notin B$; otherwise replace $B$ by $gB$ for some $g\in G$, by the minimality of  $G\curvearrowright X$ and the infinity of $X$.

Now we inductively define a sequence of neighborhoods $V_i$ of $c_i$ for each $i=1,2,3,...$, and open sets $U_{j, i}$ for each $j=1,2,...,k$ and $i=1,2,3,...$
as follows. Take pairwise disjoint open sets $V_1$ and $U_{1,1}, U_{2,1},..., U_{k,1}$ such that $c_1\in V_1$, ${\rm diam}(U_{j, 1})<1$ and $b_j\in U_{j, 1}$ for $j=1,2,...,k.$
Let $S_1=\{g\in G: gb_j\in U_{j, 1}, \forall\ j=1,2,...,k\}$. Then $S_1$ is syndetic by Theorem \ref{almost periodic}.

{\bf Claim A.} There is some $g_1\in S_1$ such that $c_2\notin \{g_1b_1, g_1b_2,..., g_1b_k\}$. Otherwise, WLOG, we may assume that $c_2=gb_1$ for all $g\in S_1$.
Noting that $e\in S_1$, $b_1$ is fixed by all $g\in S_1$. Since $S_1$ is syndetic, $Gb_1$ is finite and then $X$ is finite by the minimality of $G\curvearrowright X$.
This contradicts the infiniteness of $X$.

From Claim A, we can take pairwise disjoint open sets $V_2$ and $U_{1,2}, U_{2,2},..., U_{k,2}$ such that $c_2\in V_2$, ${\rm diam}(U_{j, 2})<1/2$ and
$g_1b_j\in U_{j, 2}\subset \overline{U_{j, 2}}\subset U_{j, 1}$ for $j=1,2,...,k.$
Applying the above discussions to $c_3$, $\{g_1b_j: j=1,2,...,k\}$, and $\{U_{j, 2}: j=1,2,...,k\}$, we get $g_2\in G$ and
pairwise disjoint open sets $V_3$ and $U_{1,3}, U_{2,3},..., U_{k,3}$ such that $c_3\in V_3$, ${\rm diam}(U_{j, 3})<1/3$ and
$g_2b_j\in U_{j, 3}\subset \overline{U_{j, 3}}\subset U_{j, 2}$ for $j=1,2,...,k.$ Going on in this way, we obtain in the end
a sequence of neighborhoods $V_i$ of $c_i$ for each $i=1,2,3,...$, and open sets $U_{j, i}$ for each $j=1,2,...,k$ and $i=1,2,3,...,$
such that for each positive $m$,
$$(\cup_{i=1}^mV_i)\cap (\cup_{j=1}^kU_{j,m})=\emptyset,  \eqno{(*)}$$
and for each $j=1,2,...,k$ and $i=1,2,3...$,
$$\overline{U_{j, i+1}}\subset U_{j, i}\ \mbox{and}\ {\rm diam}(U_{j, i})<1/i. \eqno{(**)}$$
From $(**)$, we may let $z_j=\cap_{i=1}^\infty\overline{U_{j, i}}$ for $j=1,2,...,k$. Then by the construction,
we see that $(z_1, z_2,..., z_k)\in \overline {G(b_1, b_2, ..., b_k)}\subset \overline {G(a_1, a_2, ..., a_k)}$.
By $(*)$, $\{z_1, z_2,..., z_k\}\cap C=\emptyset.$ Set $K=\{z_1, z_2,..., z_k\}$. Then $K$ satisfies the requirement.
\end{proof}

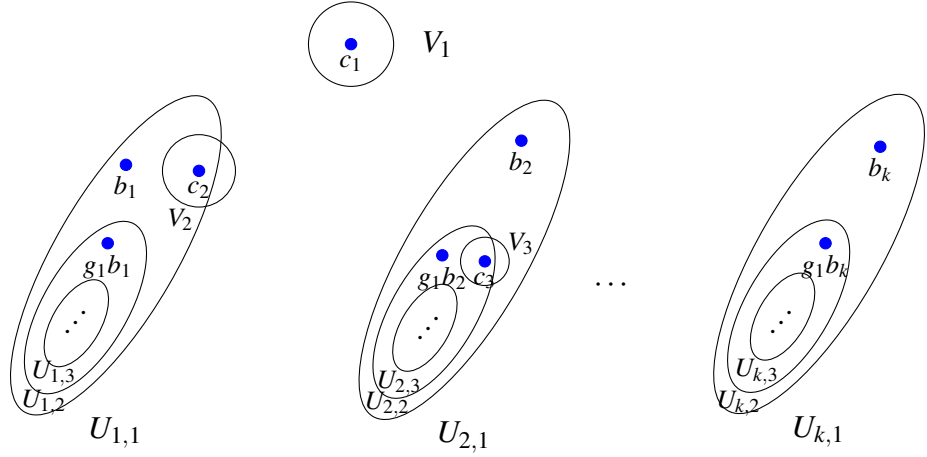
\begin{figure}
\centering
\begin{tikzpicture}[scale=0.8]

 \tikzset{
 bluenode/.style={circle, draw = blue, fill=blue, minimum size=1.5mm, inner sep=0pt},
   blacknode/.style={circle, draw = black, fill=black, minimum size=1mm, inner sep=0pt},
   } 
 \draw (3,4) circle (0.7cm) node [right=0.8cm]{$V_1$};
 \node[bluenode] (C1) at (3,4){};
 \node at (C1)[below] {\footnotesize $c_1$};
 
 \draw (-1,0) [rotate=-30] ellipse (1cm and 3cm) node [below=2cm]{$U_{1,1}$};
 \draw (-1,-1) [rotate=-30] ellipse (0.7cm and 1.6cm) node at (-0.8,-2.7){\footnotesize$U_{1,2}$};
  \draw (-1,-1.3) [rotate=-30] ellipse (0.4cm and 0.8cm) node at (-0.9,-2.2){\footnotesize$U_{1,3}$};
  \draw (0.5,1.9) circle (0.6cm) node at (0.2,1.1) {\footnotesize$V_2$};
  \node[bluenode] (C2) at (0.5,1.9){};
 \node at (C2)[below] {\footnotesize$c_2$};
   \node[bluenode] (b1) at (-0.7,2){};
 \node at (b1)[below] {\footnotesize$b_1$};
  \node[bluenode] (b1') at (-1,0.7){};
 \node at (b1')[below] {\footnotesize$g_1b_1$};
\node at (-1.5, -0.6) [rotate=50]{$\cdots$};

\draw (4,2.8) [rotate=-30] ellipse (1cm and 3cm) node [below=2cm]{$U_{2,1}$};
 \draw (4,1.8) [rotate=-30] ellipse (0.7cm and 1.6cm) node at (4.1,0.1){\footnotesize$U_{2,2}$};
  \draw (4,1.5) [rotate=-30] ellipse (0.4cm and 0.8cm) node at (4.1,0.5){\footnotesize$U_{2,3}$};
  \draw (5.2,0.4) circle (0.4cm) node at (5.8,0.7) {\footnotesize$V_3$};
  \node[bluenode] (C3) at (5.2,0.4){};
 \node at (C3)[below] {\footnotesize$c_3$};
   \node[bluenode] (b2) at (5.8,2.4){};
 \node at (b2)[below] {\footnotesize$b_2$};
  \node[bluenode] (b2') at (4.5,0.5){};
 \node at (b2')[below] {\footnotesize$g_1b_2$};
\node at (4.3, -0.7) [rotate=50]{$\cdots$};

\node at (7.3, 0) [rotate=0]{$\cdots$};

\draw (9,5.8) [rotate=-30] ellipse (1cm and 3cm) node [below=2cm]{$U_{k,1}$};
 \draw (9,4.8) [rotate=-30] ellipse (0.7cm and 1.6cm) node at (9.1,3){\footnotesize$U_{k,2}$};
  \draw (9,4.6) [rotate=-30] ellipse (0.4cm and 0.8cm) node at (9.1,3.6){\footnotesize$U_{k,3}$};
 \node[bluenode] (b3) at (11.7,2.3){};
 \node at (b3)[below] {\footnotesize$b_k$};
  \node[bluenode] (b3') at (10.8,.7){};
 \node at (b3')[below] {\footnotesize$g_1b_k$};
\node at (10.1, -0.6) [rotate=50]{$\cdots$};
\end{tikzpicture}
\caption{Illustration of the construction process}
 \end{figure}

\begin{cor}\label{margulis trick}
Let $X$ be an infinite compact metric space and let a countable group $G$ act on $X$ minimally. Then for any finite sets
$A, B\subset X$, there is a $g\in G$ such that $gA\cap B=\emptyset$.
\end{cor}
\begin{proof}
Take a countable infinite set $C\supset B$. Then applying Lemma \ref{escaping lemma} to the sets $C$ and $A$ will lead to the conclusion.
\end{proof}

Here we remark that Corollary \ref{margulis trick} is used by Margulis in \cite{Margulis} in order to construct a ping-pong-game, the proof of
which relies on a theorem due to Neumann in group theory (\cite{Neumann}). The proof here avoids using any algebraic techniques.

\section{Existence of quasi-Schottky subgroups}

Let $X$ be a connected space. A point $x\in X$ is said to be a \textit{separating point of finite order} if $X\setminus\{x\}$ has finite connected components, otherwise we say that $x$ is  a \textit{separating point of infinite order}.
\begin{lem}\label{stable component}
Let $X$ be a regular curve and $A$ be a finite set of separating points of finite order. Then $X\setminus A$ has finite components.
\end{lem}
\begin{proof}
Let $A=\{a_1,\cdots,a_n\}$ for some positive integer $n$. Suppose that the conclusion is false. Then there exists $k\in\{2,\cdots,n\}$ such that $X\setminus\{a_1,\cdots,a_{k-1}\}$ has finite components but $X\setminus\{a_1,\cdots,a_{k}\}$ has infinite components. Let $C$ denote the component of $X\setminus\{a_1,\cdots,a_{k-1}\}$ containing $a_k$. Then $C\setminus \{a_k\}$ has infinite components, saying $B_1, B_2, B_3, \cdots.$
Since $a_k$ is a separating point of finite order, there are only finitely many $i$'s such that $\overline{B_i}=B_i\cup\{a_k\}$.
Hence, there are infinitely many $i$'s with $\overline{B_i}\cap\{a_1,...,a_{k-1}\}\not=\emptyset$. Take an open neighborhood $U$ of $a_k$
such that ${\overline U}\subset X\setminus\{a_1,\cdots,a_{k-1}\}$. Then for every open set $V\subset U$ with $a_k\in V$, the boundary
$\partial_X(V)$ is infinite, which contradicts the regularity of $a_k$.
\end{proof}

\begin{lem}\cite[Theorem 1, p. 160]{Kura}\label{infinite sep pts}
There are at most countably many separating points of infinite order in a continuum.
\end{lem}
Now we are ready to prove the main theorem of the paper.

\begin{thm}\label{main theorem}
Let $G$ be a countable group and $X$ be a totally regular curve. Suppose that $\phi:G\rightarrow {\rm Homeo}(X)$ is a minimal action. Then either the action is topologically conjugate to isometries on the circle $\mathbb S^1$ (this implies that $\phi(G)$ contains an abelian subgroup of index  at most 2), or  has a quasi-Schottky subgroup (this implies that $G$ contains the free nonabelian group $\mathbb Z*\mathbb Z$).
\end{thm}

\begin{proof}
From Theorem \ref{dichotomy}, we discuss in two cases:

\medskip
{\bf Case 1.} The action $G\curvearrowright X$ is equicontinuous. Then it topologically conjugates to minimal left translations on some
homogenous space $H/K$ by Theorem \ref{euqicontinuous}, where $H$ is a compact metric topological group and $K$ is a closed subgroup of $H$.
From Theorem \ref{loc sep} and Theorem \ref{finite order pt}, $X$ has a point of finite order; and then all points of $X$ have the same order by the topological homogeneity of $X$.
Hence $X$ is a simple closed curve by Theorem \ref{same order}. Then it is a canonical fact that  $G\curvearrowright X$ is topologically
conjugates to isometries on $\mathbb S^1$ and $\phi(G)$ contains an abelian subgroup of index at most 2 (see \cite[Lemma 3]{Margulis}).

\medskip
{\bf Case 2.} The action $G\curvearrowright X$ is sensitive. By Theorems \ref{loc sep} and \ref{finite order pt}, $X$ has a point of finte order. Then, by Theorem \ref{contract neighbor}, every $x\in X$ has a contractible neighborhood.
It follows from Theorem \ref{total dim} and Definition \ref{one dim} that there is an atomless Borel probability measure $\mu$
on $X$ such that for any sequence of subcontinua $(Y_n)$ of $X$,
$$
\mu(Y_n)\rightarrow 0 \Longrightarrow \text{diam}(Y_n)\rightarrow 0. \eqno{(a)}
$$
Applying Proposition \ref{con fin} to $\mu$, there is a sequence $(g_n)$ in $G$ and a Borel probability measure $\nu$ of finite support on $X$ such that
$g_n\mu\rightarrow \nu$. Since $X$ has at most countably many separating points of infinite order by Lemma \ref{infinite sep pts}, we may further assume that
each point of ${\rm supp}(\nu)$ is a separating point of finite order. In fact, it follows from the Escaping Lemma (Lemma \ref{escaping lemma}) that $\overline{G\nu}$
always contains some element whose support consists of separating points of finite order.

Set $A={\rm supp}(\nu)=\{a_1,a_2,...,a_k\}$ for some positive integer $k$. Then for each subcontinuum $K\subset X\setminus A$,
we have
$$
0=\nu(K)\geq\limsup\limits_{n\to\infty}g_n\mu(K)=\limsup\limits_{n\to\infty}\mu(g_n^{-1}K)\geq 0.\eqno(b)
$$
Then, from (a), $\text{diam}(g_n^{-1}K)\rightarrow 0$. Let $C$ be a component of $X\setminus A$. Then $C$ is open since every component of an open set is open in Peano continuum. By Lemma \ref{peano}, for every compact subset $F$ of $C$, there is a connected compact subset $K$ satisfying $F\subset K\subset C$. Thus,
 passing to a subsequence if necessary, we may suppose $g_n^{-1}|_{C}$ converges uniformly on compact sets to a point.
This together with Lemma \ref{stable component} imply that
there is a finite set $B=\{b_1, b_2,...,b_p\}$ such that
$$g_n^{-1}|_{X\setminus A}\ \mbox{converges uniformly on compact sets to a map to}\ B. \eqno(c)$$
We may suppose $B\cap A=\emptyset$; otherwise, by Corollary \ref{margulis trick}, we can replace $B$ by $gB$ and $(g_n^{-1})$ by $(gg_n^{-1})$ for some $g\in G$.
Applying Corollary \ref{margulis trick} again, there is some $h\in G$ with
$$h(A\cup B)\cap (A\cup B)=\emptyset.$$
Set $C=hA$ and $D=hB$. Then
$$
hg_n^{-1}h^{-1}|_{X\setminus C}\ \mbox{convergences unifomaly on compact sets to a map to}\ D. \eqno(d)
$$
Now take pairwise disjoint nonempty open sets $U_1, V_1, U_2, V_2$ and $W$ such that
$$A\subset U_1, B\subset V_1, C\subset U_2,\ \mbox{and}\ D\subset V_2.$$

From (c), for sufficiently large $n$, we have $g_n^{-1}(X\setminus U_1)\subset V_1$. Thus $X\setminus g_n^{-1}(U_1)\subset V_1$, hence $X\setminus V_1\subset g_n^{-1}(U_1)$. Then $g_n(X\setminus V_1)\subset U_1$. Therefore,
$$
g_n(U_1\cup U_2\cup V_2\cup W)\subset U_1, \ \ \ g_n^{-1}(U_2\cup V_1\cup V_2\cup W)\subset V_1.
$$
Similarly from (d), for sufficiently large $n$, we have
$$
hg_nh^{-1}(U_1\cup U_2\cup V_1\cup W)\subset U_2, \ \ \ hg_n^{-1}h^{-1}(U_1\cup V_1\cup V_2\cup W)\subset V_2.
$$
Thus $g_n$ and $hg_nh^{-1}$ generate a quasi-Schottky subgroup of $G$.
\end{proof}

\end{document}